\documentclass{amsart}
\usepackage{amsmath}%
\usepackage{amsfonts}%
\usepackage{amssymb}%
\usepackage{graphicx}%
\usepackage{mathrsfs}
\newtheorem{theorem}{Theorem}
\theoremstyle{plain}

\newtheorem{definition}{Definition}

\numberwithin{equation}{section}
\begin{document}
\title[SDCPN]{ Restricted SDC Edge Cover Pebbling Number}
\author{A. Lourdusamy$^{1}$, F. Joy Beaula$^{2}$ F. Patrick$^{3}$ and I. Dhivviyanandam$^{4}$}
\address{$^{2}$Reg. No:20211282092004, $^{1,2,3}$Department of Mathematics\newline%
\indent St. Xavier's College (Autonomous), Palayamkottai - 627 002,\newline%
\indent Affiliated to Manonmaniam Sundaranar University, Abhisekapatti, Tirunelveli-627 012,\newline%
\indent Tamilnadu, India.}
\address{$^{4}$ Department of Mathematics\newline%
	\indent North Bengal St. Xavier's College, Rajganj- 735134\newline%
	\indent Affiliated to North Bengal University, Jailpaiguri, West Bengal,\newline%
	\indent  India.}
\email{$^{1}$lourdusamy15@gmail.com, $^{2}$joybeaula@gmail.com, $^{3}$patrick881990@gmail.com $^{4}$divyanasj@gmail.com}
\date{\today}

\begin{abstract}
 The restricted edge pebbling distribution is  a distribution of pebbles on the edges of $G$ is placement of pebbles on the edges with the restriction that only even number of pebbles should be  placed on the edges having labels $0$.  Given an SDC labeling of $G$, the restricted SDC edge cover pebbling number of a graph $G$, $\psi_{EC}(G)$, is the least positive integer $m$ for which  any restricted edge pebbling distribution  of $m$ pebbles such that at the end of there are no pebbles on the edges having label $0$ will allow the shifting of a pebble simultaneous to all edges with label $1$ using a sequence of restricted edge pebbling moves. We compute the restricted SDC edge cover pebbling number for some graphs.
\end{abstract}
\maketitle
\section{\bf Introduction}
A graph labeling is an assignment of integers to the vertices or
edges or both, which will satisfy certain conditions.  Lourdusamy et al. \cite{LP1} introduced  Sum divisor Cordial Labeling. Let $G$ be a graph and $g:V(G) \rightarrow \{1, 2, \cdots |V(G)|\}$ be a bijection. For each edge $ab$, assign the label $1$ if $2$ divides $(g(a)+g(b))$ and the label $0$ otherwise.  The function $g$ is called SDC labeling if $|e_{g}(0)-e_{g}(1)| \leq 1$ where $e_{g}(x), x=0,1$ is the number of edges labelled with $x$ . A graph which admits SDC labeling is called a SDC graph.

A vertex distribution of pebbles is a function form $V(G) \rightarrow N \cup \{0\}$.
A pebbling move  \cite{HB} is defined as the removal of two pebbles from some vertex and the placement of one of these pebbles on an adjacent vertex. The pebbling number \cite{FC}, $f(G)$ of a graph $G$, is the minimum number of pebbles such that regardless of their initial distribution, it is possible to move one pebble to any target vertex $v$ in $G$, using a sequence of pebbling moves. 

Edge pebbling number and cover edge pebbling number was introduced by Priscilla Paul \cite{AP}.
An edge distribution of pebbles is a function form $E(G) \rightarrow N \cup \{0\}$.
An edge pebbling move on $G$ is the process of removing two pebbles from one edge and placing one pebble on an adjacent edge. The cover edge pebbling  number $CP_{E}(G)$ of a graph $G$ is the least number of pebbles needed in a graph $G,$ so that any edge pebbling distribution of $G$  allows us on all edges of $G$ through a number of edge pebbling moves.

\section{Restricted SDC edge Cover Pebbling Number}
 
In this section we introduce a new definition Restricted SDC edge cover pebbling number.
\begin{definition} The restricted edge pebbling distribution is  a distribution of pebbles on the edges of $G$ is placement of pebbles on the edges with the restriction that only even number of pebbles should be  placed on the edges having labels $0$.  Given an SDC labeling of $G$, the restricted SDC edge cover pebbling number of a graph $G$, $\psi_{EC}(G)$, is the least positive integer $m$ for which  any restricted edge pebbling distribution  of $m$ pebbles such that at the end of there are no pebbles on the edges having label $0$ will allow the shifting of a pebble simultaneous to all edges with label $1$ using a sequence of restricted edge pebbling moves. 
\end{definition}

Consider an electrical network system in a city. The poles at the network are the nodes and the cable connecting any two poles is an edge. In this network we consider SDC labeling of this network system by assigning 1 to the cable which is under repair and assign 0 otherwise. Now restricted SDC edge cover pebbling can be thought  of a method to arrive at the cost of transportation to repair the defective cables. In the definition no pebble should remain on the edge with label 0 because there is no repair work involved.

\noindent \textbf{Note:} $p_{e}(ab)$ denotes the number of pebbles on the edge $ab$ of $G$ and $p_{e}(S)$ denotes the total number of pebbles placed on the edges in $S \subseteq G$.  $SDC_{1}$ denotes the set of edges which has label 1 by a method of SDC labeling. If the edges in $SDC_{1}$ receives  at least 1 pebble each then we say there is $SDC_{1}$ cover.

The readers can get the information about the graphs and SDC labeling of graphs $P_{n} \odot K_{1}$, $K_{1,n}$, $S(K_{1,n})$, $B_{n,n}$, $S(B_{n,n})$, $<K_{1,n}^{1} \Delta K_{1,n}^{2}>$, $DS(B_{n,n})$ and $K_{1,3} * K_{1,n}$ in \cite{LP1,LP2}
\section{Main Results}
\begin{theorem} For a comb graph $P_{n} \odot K_{1}$, $\psi_{EC}(P_{n} \odot K_{1}) = 2^{n}-2$.
\end{theorem}
\begin{proof}
From SDC labeling \cite{LP1}, we have the following edge labels:\\
$g(a_{r}a_{r+1})=1; 1 \leq r \leq n-1$\\
$g(a_{r}b_{r}=0; 1 \leq r \leq n.$\\
 If $2^{n}-4$ pebbles are placed on the edge $a_{n}b_{n}$ then covering $a_{1}a_{2}$ will use $2^{n-1}$ pebbles, covering  $a_{2}a_{3}$ will use $2^{n-2}$ pebbles $\cdots$ and covering $a_{n-2}a_{n-1}$ will use 2 pebbles. Then no pebble will remain to cover $a_{n-1}a_{n}$ which has a label 1 by SDC labeling. In order to cover $a_{n-1}a_{n}$  we need two pebbles.  Hence $\psi_{EC}(P_{n} \odot K_{1})  \geq 2^{n}-2$.

Let $D$ be a distribution of $2^{n}-2$ pebbles edges of $P_{n} \odot K_{1}$.
Let us place all the pebbles on either  $a_{n}b_{n}$ or $a_{1}b_{1}$. Without loss of generality, let us consider $p_{e}(a_{n}b_{n})=2^{n}-2$, then we can get $SDC_{1}$ cover. Let $p_{e}(P_{n} \odot K_{1})<2^{n}-2$.
  If $p_{e}(a_{n-1}a_{n})=2^{n-1}-1$, then we can get $SDC_{1}$ cover. Suppose the pendant edges $a_{r}b_{r}$ $(2 \leq r \leq n)$ receive at least  2 pebbles each we can get $SDC_{1}$ cover.
If $n$ is odd, $r$ is even and if the edges $a_{r}b_{r}$, $2 \leq r \leq n$   have at least  4 pebbles  each then we can  get  $SDC_{1}$. For $n$ is even $r$ is even and the edges $a_{r}b_{r}$, $2 \leq r \leq n-2$   have at least 4 pebbles each and the edge $a_{n}b_{n}$ has at least 2 pebbles then we can get $SDC_{1}$ cover. Placing 3 pebbles on $a_{n-1}a_{n}$ and 2 pebbles each on $n-3$ pendant edges  we are done. So $\psi_{EC}(P_{n} \odot K_{1}) \leq  2^{n}-2$.

Therefore  $\psi_{EC}(P_{n} \odot K_{1}) = 2^{n}-2$.

\end{proof}
\begin{theorem} For a star graph $K_{1,n}$,  $\psi_{EC}(K_{1,n}) = \begin{cases} n-1 & \text{n odd} \\ n & \text{n even}\end{cases}$.
\end{theorem}
\begin{proof}
From the  SDC labeling \cite{LP1} we have the following edge labels:\\
$g(aa_{r})=0;$ if $r$ is odd,\\
$g(aa_{r}=1.$ if $r$ is even.\\
\textbf{Case 1.} $n$ is odd.\\
If  we place $n-3$ pebbles on $aa_{1}$ then we cannot cover an edge $aa_{n-1}$. In order to cover an edge $aa_{n-1}$ we need two more pebbles. Hence $\psi_{EC}( K_{1,n})  \geq n-1$.

For proving the sufficient condition we distribute $n-1$ pebbles on $E(K_{1,n})$.
In the SDC labeling there are $\left\lfloor \frac{n}{2} \right\rfloor$ edges having label $1$. Then we have to cover these edges using $n-1$ pebbles on the edge $aa_{r}$, $r$ is odd. Using $n-2$ pebbles on  $aa_{r}$, $(r$ is even) we get $SDC_{1}$ cover. Placing 2 pebbles each on  $aa_{r}$,  $(1 \leq r \leq n-2$  and $r$ is odd),  we can get $SDC_{1}$ cover.  Placing 1 pebble each on the edge $aa_{r}$, ($r$ is even) we are done.

 \textbf{Case 2.} $n$ is even.\\
Placing  $n-2$ pebbles on $aa_{1}$  we cannot cover  $aa_{n}$. In order to cover  $aa_{n}$ we need two more pebbles.  Hence $\psi_{EC}( K_{1,n})  \geq n$.

For proving the sufficient condition we distribute $n$ pebbles on $E(K_{1,n})$.
By  SDC labeling there are $ \frac{n}{2}$ edges having label $1$.
Now the aim is to place one pebble each on $SDC_{1}$ set. In order to cover   $ \frac{n}{2}$ edges of $SDC_{1}$ set we have to place $n$ pebbles  $aa_{r}$, ($1 \leq r \leq n$ and $r$ is odd).
Similarly using $n-1$ pebbles on the edge $aa_{r}$,  ($1 \leq r \leq n$ and  $r$ is even) then we can get $SDC_{1}$ cover. Placing 2 pebbles on the edge  $aa_{r}$,  $1 \leq r \leq n-1$  and $r$ is odd, then we can cover $SDC_{1}$. By placing 1 pebble each on $aa_{r}$, ($1 \leq r \leq n$ and $r$ is even) we are done.
 Thus  $\psi_{EC}(K_{1,n}) = \begin{cases} n-1 & \text{n odd} \\ n & \text{n even}\end{cases}$.
\end{proof}

\begin{theorem} For a  graph $S(K_{1,n})$, $\psi_{EC}(S(K_{1,n})) = 4n-2$.
\end{theorem}
\begin{proof}
In SDC labeling \cite{LP1} we have the following pattern of SDC edge labels:\\
$g(aa_{r})=1$, $1 \leq r \leq n$;\\
$g(a_{r}b_{r})=0$, $1 \leq r \leq n$.

If $p_{e}(a_{n}b_{n})=4n-4$  then we cannot cover $aa_{1}$ which has label 1. In order to cover $aa_{1}$ we need two more pebbles. Hence $\psi_{EC}(S(K_{1,n})) \geq 4n-2$.

For proving the sufficient condition we distribute $4n-2$ pebbles on $E(S(K_{1,n}))$.
In SDC labeling we have $n$  edges  with labels 1. In order to cover these edges we use  $4n-2$ pebbles on  $aa_{r}$, $(1 \leq r \leq n)$.
Let $p_{e}(S(K_{1,n}))< 4n-2$.
 Suppose $p_{e}(a_{r}b_{r}) \geq 2,$ $(\forall \ 1 \leq r \leq n)$. Then we can get  $SDC_{1}$ cover. If  we place $2n-1$ pebbles on  any one the edges of $SDC_{1}$ set then we are done. Suppose we distribute pebbles on the edges of $SDC_{0}$ set and $SDC_{1}$  set we use less than or equal to $2n-1$ pebbles. Placing 1 pebble each on the edge $aa_{r},$ $(1 \leq r \leq n)$ we get $SDC_{1}$ cover. 
Hence $\psi_{EC}( S(K_{1,n})) \leq 4n-2$.
\end{proof}

\begin{theorem} 
For a Bistar graph $B_{n,n}$, $\psi_{EC}(B_{n,n}) = \begin{cases} 3n+3 & \text{n odd} \\ 3n & \text{n even}\end{cases}$.
\end{theorem}
\begin{proof}
In SDC labeling \cite{LP1} we have the following pattern of edge labels:\\
$g(ab)=0$;\\
$g(aa_{r})=\begin{cases} 1 & \text{r is odd} \\ 0 & \text{ r is even}\end{cases}$, $1 \leq r \leq n$;\\
$g(bb_{r})=\begin{cases} 1 & \text{r is odd} \\ 0 & \text{ r is even}\end{cases}$, $1 \leq r \leq n$.\\

\textbf{Case 1.} $n$ is odd.\\
Placing $3n+1$ pebbles on $bb_{n-1}$, then we cannot cover the edge $aa_{1}$ which has label 1 by SDC labeling.  Hence $\psi_{EC}(B_{n,n})  \geq 3n+3$.

For proving the sufficient condition we distribute $3n+3$ pebbles on $E(B_{n,n})$.
From SDC labeling method  $n+1$ edges  receive  label 1. In order to get $SDC_{1}$ cover we use $3n+3$ pebbles on the edge $aa_{r}$ or $bb_{r}$ if $r$ is even.
Let $p_{e}(B_{n,n})<3n+3$.
$SDC_{1}$ cover is ensured for the following distributions:\\
for $ 1 \leq r \leq n$
\begin{enumerate}
\item $p_{e}(aa_{r})=3n+2$ or $p_{e}(bb_{r})=3n+2$, ($r$ is odd)
\item $p_{e}(ab)=2n+2$
\end{enumerate}
  If we place $n+1$ pebbles on $aa_{r}$ or $bb_{r}$, $1 \leq r \leq n$ where $r$ is even and  $n+1$ pebbles on $ab$ then we reach $SDC_{1}$ cover. If $p_{e}(aa_{r})=n$ and $p_{e}(bb_{r})=n$, $(1 \leq r \leq n$ and $r$ is odd) then we reach $SDC_{1}$ cover.  If $p_{e}(aa_{r})=\frac{n-1}{2}$, $p_{e}(bb_{r})=\frac{n-1}{2}$, $(1 \leq r \leq n$ and $r$ is even) and $p_{e}(ab)=4$ then we get  $SDC_{1}$ cover. We have $\left\lceil \frac{n}{2} \right\rceil$  edges with label 1  that are incident  with each of two apex vertices. With $n+1$ pebbles  distributed on each set of $\left\lceil \frac{n}{2} \right\rceil$ edges that are incident  with each of two apex vertex we reach $SDC_{1}$ cover.

\textbf{Case 2.} $n$ is even.\\
If $p_{e}(bb_{n})=3n-2$,  then we cannot cover the edge $aa_{1}$ which has label 1 by SDC labeling .  Hence $\psi_{EC}(B_{n,n})  \geq 3n$.

For proving the sufficient condition we distribute $3n$ pebbles on $E(B_{n,n})$.
From the  SDC labeling pattern we have $n$ number of edges labeled with 1. In order to get $SDC_{1}$ cover  we place $3n$ pebbles on the edge $aa_{r}$ or $bb_{r}$, ($ 1 \leq r \leq n$ $r$ is even) then we are done.
Let $p_{e}(B_{n,n})<3n$.
$SDC_{1}$ cover is ensured for the following distributions:\\
for $ 1 \leq r \leq n$
\begin{enumerate}
\item $p_{e}(aa_{r})=3n-1$, or $p_{e}(bb_{r})=3n-1,$($r$ is odd)
\item $p_{e}(ab)=2n$
\end{enumerate} 
 If we place $n$ pebbles on $aa_{r}$ or $bb_{r}$, $1 \leq r \leq n$ where $r$ is even and  n pebbles on $ab$ then we get  $SDC_{1}$ cover. If $p_{e}(aa_{r})=n-1$ and $p_{e}(bb_{r})=n-1$, ($ 1 \leq r \leq n$ and $r$ is odd) then we reach $SDC_{1}$ cover.  If $p_{e}(aa_{r})=\frac{n}{2}$, $p_{e}(bb_{r})=\frac{n}{2}$,  ($ 1 \leq r \leq n$ and $r$ is even) and $p_{e}(ab)=n$ then we get $SDC_{1}$ cover.  We have $\frac{n}{2}$  edges with label 1  that are incident  with each of two apex vertices. With $n$ pebbles  distributed on each set of $ \frac{n}{2}$ edges that are incident  with each of two apex vertex we reach $SDC_{1}$ cover.\\
 Hence, $\psi_{EC}(B_{n,n}) = \begin{cases} 3n+3 & \text{n odd} \\ 3n & \text{n even}\end{cases}.$
\end{proof}

\begin{theorem} 
For a graph $S(B_{n,n})$, $\psi_{EC}(S(B_{n,n})) = 20n+6$.
\end{theorem}
\begin{proof}
In SDC labeling \cite{LP2} we have the following pattern of SDC edge labels:\\
$g(aa^{'})=1$;\\
$g(a^{'}b)=0$;\\
$g(aa_{r}^{'})=1$, $1 \leq r \leq n$;\\
$g(a_{r}^{'}a_{r})=0, 1 \leq r \leq n$;\\
$g(bb_{r}^{'})=1$, $1 \leq r \leq n$;\\
$g(b_{r}^{'}b_{r})=0, 1 \leq r \leq n$.\\

If $p_{e}(b_{1}^{'}b_{1})=20n+4$,  then we cannot cover the edge $bb_{1}^{'}$ which has  label 1 by SDC labeling. In order to cover the edges we require two additional  pebbles. Hence $\psi_{EC}(S(B_{n,n}))  \geq 20n+6$.

For proving the sufficient condition we distribute $20n+6$ pebbles on $E(S(B_{n,n}))$.
If $p_{e}(b_{r}^{'}b_{r})=20n+6$, ($1 \leq r \leq n$) then we reach $SDC_{1}$ cover.
Let $p_{e}(S(B_{n,n}))< 20n+6$.
If there is  an edge in $\{a_{r}^{'}a_{r} : 1 \leq r \leq n\}$ with $20n+2$ pebbles then we reach $SDC_{1}$ cover. If $p_{e}(bb_{r}^{'})=10n+3$, ($1 \leq r \leq n$) then we can cover the $SDC_{1}$.  If $p_{e}(aa_{r}^{'})=10n+1$, ($1 \leq r \leq n$) then we can cover the $SDC_{1}$.  Using  $6n+2$ pebbles on $a^{'}b$  we get  $SDC_{1}$ cover. If $p_{e}(aa^{'})=6n+1$ then  $SDC_{1}$ cover is obtained. Placing $4n$ pebbles on any one of the edges in $\{a_{r}^{'}a_{r} : 1 \leq r \leq n\}$ and $4n-2$  pebbles on any one of the edge in $\{b_{r}^{'}b_{r} : 1 \leq r \leq n\}$  we reach $SDC_{1}$ cover. For $1 \leq r \leq n$, if $p_{e}(aa_{r}^{'})=2n+1$ and $p_{e}(bb_{r}^{'})=2n-1$  then  $SDC_{1}$ cover is ensured. Suppose we place 2 pebbles on $a^{'}b$  and $4n-2$ pebbles   on any one of the edges in $\{a_{r}^{'}a_{r} : 1 \leq r \leq n\}$ and  $4n-2$ pebbles  on any one of the edges in $\{b_{r}^{'}b_{r} : 1 \leq r \leq n\}$ then we  can reach $SDC_{1}$ cover. For  $1 \leq r \leq n$, if we  place 2 pebbles on $a^{'}b$  and $2n-1$ pebbles on  $aa_{r}^{'}$ and  $2n-1$ pebbles on $bb{r}^{'}$ then we can reach $SDC_{1}$ cover.  If we place $2n+1$ pebbles on $aa^{'}$ and 1 pebble each on $b_{r}^{'}b_{r},$ $1 \leq r \leq n$ we obtain $SDC_{1}$ cover. Put $4n+1$ pebbles on $aa^{'}$ and 2 pebbles each on $a_{r}^{'}a_{r}$ ($ 1 \leq r \leq n$)  then we  get $SDC_{1}$ cover. If $p_{e}(aa^{'})=2n+1$ and $p_{e}(a^{'}b)=2n$ then we reach $SDC_{1}$ cover. Placing at least 2 pebbles on each of the pendant edges we can cover the edges $bb_{r}^{'}$ and $aa_{r}^{'}$ $(1 \leq r \leq n)$. Place  4 pebbles additionally on one of the edges in   $\{aa_{r}^{'}: 1 \leq r \leq n\}$ to move a pebble to $aa^{'}$. Thus we are done.\\
Therefore $\psi_{EC}(S(B_{n,n}))  = 20n+6$.

\end{proof}

\begin{theorem} 
For $<K_{1,n}^{(1)} \Delta K_{1,n}^{(2)}>$, $\psi_{EC}(<K_{1,n}^{(1)} \Delta K_{1,n}^{(2)}>) = \begin{cases} 3n+3 & \text{n odd} \\ 3n+2 & \text{n even}\end{cases}$.
\end{theorem}
\begin{proof}
In SDC labeling \cite{LP2} we have the following pattern of SDC edge labels:\\
$g(ax)=0$;\\
$g(bx)=0$;\\
$g(ab)=1$;\\
$g(a^{'}b)=0$;\\
$g(aa_{r})=\begin{cases} 0 & \text{r is odd} \\ 1 & \text{ r is even}\end{cases}$, $1 \leq r \leq n$;\\
$g(bb_{r})=\begin{cases} 1 & \text{r is odd} \\ 0 & \text{ r is even}\end{cases}$, $1 \leq r \leq n$;\\

\textbf{Case 1.} $n$ is odd.\\
If $p_{e}(aa_{1})=3n+1$,  then we cannot cover an edge $ab$ which has  label 1.  Hence $\psi_{EC}(<K_{1,n}^{(1)} \Delta K_{1,n}^{(2)}>) \geq 3n+3$.

For proving the sufficient condition we distribute $3n+3$ pebbles on $E(<K_{1,n}^{(1)} \Delta K_{1,n}^{(2)}>)$.
If $p_{e}(aa_{r})=3n+3$ or $p_{e}(ax)=3n+3$ ($ 1 \leq r \leq n$ and $r$ is odd) then we get $SDC_{1}$ cover. Suppose $p_{e}(<K_{1,n}^{(1)} \Delta K_{1,n}^{(2)}>) <3n+3$.\\
$SDC_{1}$ cover is ensured for the following distributions:\\
for $ 1 \leq r \leq n$
\begin{enumerate}
\item $p_{e}(aa_{r})=3n+2$, ($r$ is even)
\item $p_{e}(bb_{r})=3n+1$ or $p_{e}(bx)=3n+1$, ($r$ is even)
\item $p_{e}(bb_{r})=3n$, ($r$ is odd)
\item $p_{e}(ab)=2n+1$
\end{enumerate}
 Placing $n+1$ pebbles on $aa_{r}, r$ is odd  and $n+1$ pebbles on $bb_{r}, r$ is even then we can reach $SDC_{1}$ cover. If we place 2 pebbles each on the pendant edges having label o then  $SDC_{1}$ cover is ensured.
If $p_{e}(ax)=n+1$ and $p_{e}(bx)=n+1$ then we can easily reach $SDC_{1}$ cover. Placing $n$ pebbles on $aa_{r}, r$ is even  and $bb_{r}, r$ is odd will get as $SDC_{1}$ cover. If $p_{e}(ab)=n$ and $p_{e}(bx)=n+1$ then we reach $SDC_{1}$ cover. 
Thus $\psi_{EC}(<K_{1,n}^{(1)} \Delta K_{1,n}^{(2)}>) = 3n+3$.

\textbf{Case 2.} $n$ is even.\\
If $p_{e}(aa_{1})=3n$,  then we cannot cover the edge $ab$ which has  label 1.  Hence $\psi_{EC}(<K_{1,n}^{(1)} \Delta K_{1,n}^{(2)}>) \geq 3n+2$.

For proving the sufficient condition we distribute $3n+2$ pebbles on $E(<K_{1,n}^{(1)} \Delta K_{1,n}^{(2)}>)$.
If $p_{e}(aa_{r})=3n+2$ ($r$ is odd) or $p_{e}(ax)=3n+2$ $p_{e}(bb_{r})=3n+2$ ($r$ is even) or $p_{e}(bx)=3n+2$ ,  then we can cover the edges which have label 1. Suppose $p_{e}(<K_{1,n}^{(1)} \Delta K_{1,n}^{(2)}>) <3n+2$. 
$SDC_{1}$ cover is ensured for the following distributions:\\
for $ 1 \leq r \leq n$
\begin{enumerate}
\item $p_{e}(aa_{r})=3n+1$, ($r$ is even) or $p_{e}(bb_{r})=3n+1$ ($r$ is odd)
\item $p_{e}(ab)=2n+1$
\end{enumerate}
  Placing $n+2$ pebbles on $aa_{r}, r$ is odd  and $bb_{r}, r$ is even then we can get $SDC_{1}$ cover. Placing 2 pebbles each on the pendant edges having label 0  and 2 pebbles on the edge $ax$ then we reach $SDC_{1}$ cover. If $p_{e}(ax)=n+2$ and $p_{e}(bx)=n$ then we can easily obtain $SDC_{1}$ cover. Placing $n+1$ pebbles on $aa_{r}, r$ is even  and $n-1$ pebbles on $bb_{r}, r$ is odd leads to $SDC_{1}$ cover. If $p_{e}(ab)=n$ and $p_{e}(bx)=n$ it is easy to get $SDC_{1}$ cover . 
Thus $\psi_{EC}(<K_{1,n}^{(1)} \Delta K_{1,n}^{(2)}>) = 3n+2$.
\end{proof}

\begin{theorem} 
For a graph $DS(B_{n,n})$, $\psi_{EC}(DS(B_{n,n})) = 8n+2$.
\end{theorem}
\begin{proof}
In  SDC labeling \cite{LP2} we have the following edge labels:\\
$g(uv)=1;$\\
$g(uw_{2})=0;$\\
$g(vw_{2})=0;$\\
$g(uu_{r})=1,$  $1 \leq r \leq n$;\\
$g(u_{r}w_{1})=0,$  $1 \leq r \leq n$;\\
$g(vv_{r})=0,$  $1 \leq r \leq n$;\\
$g(v_{r}w_{1})=1,$  $1 \leq r \leq n$.\\

If $p_{e}(vw_{2})=8n$,  then we cannot cover the edge $uv$ which has label 1. In order to cover an edge $uv$  we need two more pebbles. Hence $\psi_{EC}(DS(B_{n,n}))  \geq 8n+2$.

For proving the sufficient condition we distribute $8n+2$ pebbles on $E(DS(B_{n,n}))$.
 If $p_{e}(vw_{2})=8n+2$, then we are sure to get $SDC_{1}$ cover. Let $p_{e}(vw_{2}) < 8n+2$.

$SDC_{1}$ cover is ensured for the following distributions:\\
for $ 1 \leq r \leq n$
\begin{enumerate}
\item $p_{e}(vv_{r})=8n$ 
\item $p_{e}(v_{r}w_{1})=6n+3$ 
\item $p_{e}(uw_{2})=6n+2$
\item $p_{e}(uu_{r})=6n+1$ or $p_{e}(uv)=6n+1$
\item $p_{e}(u_{r}w_{1})=4n+4$
\end{enumerate}
  Distributing 2 pebbles each on the edges $vv_{r}$  and $u_{r}w_{1},$ $(1 \leq r \leq n)$ we can cover $uu_{r}$ and $v_{r}w_{1}$ $(1 \leq r \leq n)$. So we can cover the edge $uv$ using 2 pebbles on $vw_{2}$. If $p_{e}(uw_{2})=2n+2$ and $p_{e}(u_{r}w_{1})=2n$ ($1 \leq r \leq n$) then it is easy to see $SDC_{1}$ cover. If $p_{e}(uv)=2n+1$ and $p_{e}(vv_{r})=2n$ ($1 \leq r \leq n$) then  $SDC_{1}$ cover is ensured. Place $2n+1$ pebbles on the edge $uv$ and  $2n-1$ pebbles on the edge $v_{r}w_{1}$  ($1 \leq r \leq n$) then also $SDC_{1}$ cover is reached. If $p_{e}(u_{r}w_{1})=4n$ ($1 \leq r \leq n$) and $p_{e}(vw_{2})=2$ then  $SDC_{1}$ cover is obtained.  If $p_{e}(vw_{2})=4n+2$ and $p_{e}(v_{r}w_{1})=n+1$ then it is easy to see $SDC_{1}$ cover. 

Therefore, $\psi_{EC}(DS(B_{n,n}))  = 8n+2$.

\end{proof}

\begin{theorem} 
For $K_{1,3} * K_{1,n}$, $\psi_{EC}(K_{1,3} * K_{1,n}) = \begin{cases} 9n+11 & \text{n odd} \\ 9n+4 & \text{n even}\end{cases}$.
\end{theorem}
\begin{proof}
In SDC labeling \cite{LP1}  we have the following pattern of SDC edge labels:\\
$g(xu)=1$;\\
$g(xv)=0$;\\
$g(xw)=0$;\\
$g(uu_{r})=\begin{cases} 1 & \text{r is odd} \\ 0 & \text{ r is even}\end{cases}$, $1 \leq r \leq n$;\\
$g(vv_{r})=\begin{cases} 1 & \text{r is odd} \\ 0 & \text{ r is even}\end{cases}$, $1 \leq r \leq n$;\\
$g(ww_{r})=\begin{cases} 0 & \text{r is odd} \\ 1 & \text{ r is even}\end{cases}$, $1 \leq r \leq n$.\\

\textbf{Case 1.} $n$ is odd.\\
If $p_{e}(ww_{1})=9n+9$,  then we cannot cover the edge $ww_{2}$ which has label 1.  Hence $\psi_{EC}(K_{1,3} * K_{1,n}) \geq 9n+11$.

For proving the sufficient condition we distribute $9n+11$ pebbles on $E(K_{1,3} * K_{1,n})$.
If $p_{e}(ww_{2r-1})=9n+11$ ($1 \leq r \leq \left\lceil \frac{n} {2} \right\rceil$) then we can easily cover the edges having label 1. Let $p_{e}(K_{1,3} * K_{1,n})<9n+11$.
$SDC_{1}$ cover is ensured for the following distributions:
\begin{enumerate}
\item $p_{e}(ww_{2r})=9n+10$ ($1 \leq r \leq \left\lfloor \frac{n} {2} \right\rfloor$)
\item $p_{e}(vv_{2r})=9n+5$  ($1 \leq r \leq \left\lfloor \frac{n} {2} \right\rfloor$) 
\item $p_{e}(vv_{2r-1})=9n+4$ ($1 \leq r \leq \left\lceil \frac{n} {2} \right\rceil$)
\item $p_{e}(uu_{2r})=9n+3$ ($1 \leq r \leq \left\lfloor \frac{n} {2} \right\rfloor$)
\item $p_{e}(uu_{2r-1})=9n+2$ ($1 \leq r \leq \left\lceil \frac{n} {2} \right\rceil$)
\item $p_{e}(xu)=5n+2$
\item $p_{e}(xv)=5n+3$
\item  $p_{e}(xw)=5n+5$
\end{enumerate}
If we place $n+1$ pebbles each on $ux$, $vx$ and $wx$ then we obtain $SDC_{1}$ cover. If we place $n+3$ pebbles on $uu_{r}$, $r$ is odd, $n+1$ pebbles on $vv_{r}$, $r$ is odd and $n-1$ pebbles on $ww_{r}$, $r$ is even then we can cover $SDC_{1}$. Placing $n+2$ pebbles  on $uu_{r}$, $r$ is even, $n$ pebbles on $vv_{r}$, $r$ is even and $n-2$ pebbles on $ww_{r}$, $r$ is odd then we reach  $SDC_{1}$ cover. 
 Let us consider the distribution of  pebbles on the edges having label 0  except the edge $xv$. If we place $n+1$ pebbles are placed on $uu_{r}$, $r$ is odd, 2 pebbles on $xw$, $n+1$ pebbles on $vv_{r}$, $r$ is odd and $n-1$ pebbles on $ww_{r}$, $r$ is even then we reach $SDC_{1}$ cover.  If we place $n$ pebbles on $uu_{r}$, $r$ is even, 2 pebbles on $xw$, $n+1$ pebbles on $vv_{r}$, $r$ is even and $n$ pebbles on $ww_{r}$, $r$ is odd,  then we are sure to get $SDC_{1}$ cover. 

Thus $\psi_{EC}(K_{1,3} * K_{1,n}) = 9n+11$.

\textbf{Case 2.} $n$ is even.\\
If $p_{e}(ww_{1})=9n+2$,  then we cannot cover an edge $ww_{2}$ which has a label 1 by SDC condition.  Hence $\psi_{EC}(K_{1,3} * K_{1,n}) \geq 9n+4$.

For proving the sufficient condition we distribute $9n+4$ pebbles on $E(K_{1,3} * K_{1,n})$.
If $p_{e}(ww_{2r-1})=9n+4$ ($1 \leq r \leq \left\lceil \frac{n} {2} \right\rceil$) or $p_{e}(vv_{2r})=9n+4$, ($1 \leq r \leq \left\lfloor \frac{n} {2} \right\rfloor$), we can easily cover $SDC_{1}$.  Let $p_{e}(K_{1,3} * K_{1,n})<9n+4$.
$SDC_{1}$ cover is ensured for the following distributions:
\begin{enumerate}
\item $p_{e}(ww_{2r})=9n+3$ ($1 \leq r \leq \left\lfloor \frac{n} {2} \right\rfloor$) or $p_{e}(vv_{2r-1})=9n+3$ ($1 \leq r \leq \left\lceil \frac{n} {2} \right\rceil$)
\item  $p_{e}(uu_{2r})=9n+2$  ($1 \leq r \leq \left\lfloor \frac{n} {2} \right\rfloor$)
\item $p_{e}(uu_{2r-1})=9n+1$ ($1 \leq r \leq \left\lceil \frac{n} {2} \right\rceil$)
\item $p_{e}(xu)=5n+1$
\item $p_{e}(xv)=5n+2$
\item  $p_{e}(xw)=5n+2$
\end{enumerate}
  If we place $n$ pebbles each on  the set $uu_{r}$, $vv_{r}$ ($r$ is odd) and $ww_{r}$ ($r$ is even)  then we can cover all the edges having label 1 except the edge $xu$. In order to cover the edge $xu$ we put 2 pebbles on  $xw$. If we place  $n-1$ pebbles each on on the set $uu_{r}$, $vv_{r}$ ($r$ is even) and $ww_{r}$, ($r$ is odd) then we are able to cover all the edges except the edge  $xu$.   But the edge $xu$ can be covered by  placing 2 pebbles on the edge $xw$. If we place $n$ pebbles each on  $vx$ and $wx$ and $n+1$ pebbles on $ux$ then we reach $SDC_{1}$ cover. If we place $n+2$ pebbles are placed on $uu_{r}$ ($r$ is odd), $n$ pebbles each on the set $vv_{r}$ ($r$ is odd) and $ww_{r}$ ($r$ is even) then we get $SDC_{1}$ cover. Placing $n+1$ pebbles  on $uu_{r}$ ($r$ is even), $n-1$ pebbles each on the set $vv_{r}$ ($r$ is even) and $ww_{r}$, $r$ is odd ensures $SDC_{1}$ cover. 
 
Thus $\psi_{EC}(K_{1,3} * K_{1,n}) = 9n+4$.
\end{proof}


\end{document}